 \documentclass{amsart}

\usepackage{graphicx}
\usepackage{amsthm}
\usepackage{amsmath}

\usepackage[utf8]{inputenc}
\usepackage[T1]{fontenc}

\usepackage{amssymb}

 \usepackage{amsaddr}

\newtheorem{deff}{Definition}[section]
\newtheorem{twr}{Theorem}

\newtheorem{remark}[deff]{Remark}

\newcommand{\comment}[1]{}

\newcommand{\e}[0]{\mathbf{e}}
\newcommand{\PP}[0]{\mathbf{P}}

\newcommand{\E}[0]{\mathbb{E}}

\newcommand{\C}[0]{\mathbf{C}}

\newcommand{\es}[0]{\mathbf{s}}

\newcommand{\EE}{\mathbb E}

\usepackage{tikz}

\begin{document}

\title[Generalized Gambler's ruin problem via Siegmund duality]{Generalized Gambler's ruin problem: explicit formulas via Siegmund duality
}
\thanks{Work supported by NCN Research Grant DEC-2013/10/E/ST1/00359}


\author{Pawe{\l} Lorek  } 
\address{Mathematical Institute, University of Wroc{\l}aw, Poland}
\email{Pawel.Lorek@math.uni.wroc.pl}




\date{\today}
\maketitle

\begin{abstract}
We give explicit formulas for ruin probabilities in  a multidimensional Generalized  Gambler's ruin problem. 
The generalization is best interpreted as a game of one player against $d$ other players,
allowing arbitrary  winning and losing probabilities (including ties) depending on the current fortune with particular player.
It includes many previous other generalizations as special cases. Instead of usually utilized first-step-like analysis we involve dualities between Markov chains.
We give general procedure for solving ruin-like problems utilizing Siegmund duality in Markov chains for partially ordered 
state spaces studied recently  in context of M\"obius monotonicity.
\end{abstract}

\section{Introduction}
Gambler's ruin problem has been playing important role in applied mathematics.
There are applications in some casino games, e.g. craps \cite{Isaac1995}, blackjack \cite{Snell2009},
physics \cite{El-Shehawey2000}, \cite{Yamamoto2013}, hydrology  \cite{Tsai2014},
biology and epidemic models \cite{Harik1999}, finance \cite{Scott1981}, \cite{Rolski2009}, \cite{Asmussen2010},
just to mention few. There are many variations of the problem, newer ones are formulated after older ones are solved.
For example, the following variations were proposed: infinite amount of money, three or more players \cite{Kmet2002}, \cite{Rocha2004}, 
 the \textsl{attrition variation} (applies, e.g., to World Series or Stanley Cup finals \cite{Kaigh1979}),
 some cases of winning probabilities being dependent on current fortune \cite{El-Shehawey2009}, \cite{Lefebvre2008}.
\par 
The problem  is as relevant today as it was in 17th century. According to \cite{Edwards1983}
Pascal was the first who posed a problem in 1656 in a letter to Fermat. The most common form comes 
from Huygens who restated the problem as follows (rephrased):
 \begin{center}
  \begin{minipage}[t]{0.85\textwidth}
\textsl{``Let two men play with three dice, the first player scoring a point whenever 11 is thrown,
whereas the second whenever 14 is thrown. Each player starts with 12 points. Successful roll adds one point to the player and subtracts one from the other player.
The loser of the game is the first to reach zero points. What is the probability of victory for each player?
``}
\end{minipage}  
\end{center}
Huygens gave a solution of above problem. James Bernoulli \cite{Bernoulli1713} in 1713 
generalized and replaced Huygens' numerical results by formulas, i.e., he considered general initial capital $i$, general
total amount of money $N$ and general winning probabilities $p\in(0,1)$. Since then the problem became very popular
and different proofs were obtained. Pascal  gave his solution to Cercavi without mentioning the method.
According to Edwards \cite{Edwards1983} Fermat probably used the combinatorial argument (as in ''Problem of points``),
most methods which appeared later used some kind of first-step-like analysis which led to solving some recursions.

It is known that the absorption probability of given chain can be related to the stationary distribution of some other ergodic chain.
The relation is given via so-called \textsl{Siegmund duality}, the notion introduced in \cite{Siegmund1976}. 
It was studied in financial context, where the probability that a dual risk process starting at level $h$ is ruined equals the probability that the 
stationary queue length exceeds level $h$, see \cite{Asmussen2010}, \cite{Asmussen2009}.
Already in \cite{Lindley1952a} such duality between some random walks on integers was shown. For this duality reader is also 
referred to \cite{Theodore1983},  \cite{Diaconis1990a}, \cite{Dette1997a}  or \cite{Huillet2010}, just to mention few.
All above papers have one thing in common: they study Siegmund duality defined for linear ordering of the state space (and most of them
 birth and death chains only).
In this case  \cite{Siegmund1976} states that the process has such dual if and only if it is stochastically
monotone (w.r.t. total ordering). It is a little bit surprising that it was not exploited in the context of one-dimensional Gambler's ruin problem.
The solutions of the  classical problem and it's various one-dimensional generalizations are special cases of  Theorem \ref{twr:twr_main} and can be relatively easy calculated using usual stochastic monotonicity.
On the other hand  multidimensional case is quite different. For partial nonlinear ordering the stochastic monotonicity does not imply
the existence of Siegmund dual, see  \cite{Liggett2004}. Finding such duals was successful for some specific chains and/or orderings.
For example, in financial context, in \cite{Blaszczyszyn1999} authors considered $R^d$-valued Markov processes  (their Siegmund dual was set-valued).
Recently, \cite{Huillet2014b}  considered dualities  for  Markov chains on partitions and sets.
In \cite{Lorek2015_Siegmund_duality} we show that Siegmund dual exists if and only if chain is M\"obius monotone,
the connections with Strong Stationary Duality (consult \cite{Diaconis1990a}) is also given therein. 
Let us mention at this point that for non-linear ordering M\"obius and stochastic monotonicities are, in general, different.
In particular, we can have a chain which is not stochastically monotone, but which is M\"obius monotone, thus we are able to construct 
it's Siegmund dual.

In this paper, based on results from \cite{Lorek2015_Siegmund_duality}, we give the solution to the multidimensional
Generalized Gambler's ruin problem. The paper is organized as follows. 
In Section \ref{sec:main_result} we describe our Generalized Gambler's ruin problem, state it's solution (Theorem \ref{twr:twr_main})  and point out other results as  some special cases.
In Section \ref{sec:siegmund} we recall notion of Siegmund duality and antiduality for chains on partially ordered state spaces 
and give a general recipe for calculating ruin-like probabilities (summarized in Theorem \ref{th:main_antiduality}). Section 
\ref{sec:toy_example} includes some toy example (\textsl{Cat Eats Mouse Eats Cheese}), where the case of negative antidual matrix is presented. 
Finally, Section \ref{sec:proof_of_main} contains proof of Theorem \ref{twr:twr_main}.

\section{Generalized Gambler's ruin problem and main result}\label{sec:main_result}

In the one-dimensional Gambler's ruin problem two players start a game with total amount of, say, $N$ dollars and initial values   $k$ and $N-k$.
At each step they flip the   coin (not necessary unbiased) to decide who wins a dollar. 
The game is over when one of them goes bankrupt.\par 
We will consider the following generalization.  There is one player (referred as ``we'') playing with $d\geq 1$ other players.
Our initial assets are $(i_1,\ldots,i_d)$ and assets of consecutive players are $(N_1-i_1,\ldots,N_d-i_d)$
($N_j\geq 1$ is a total amount of assets with player $j$).
Then, with probability $p_j(i_j)$ we win one dollar with player $j$ and with probability $q_j(i_j)$ we lose it.
With the remaining probability $1-\sum_{k=1}^d(p_k(i_j)+q_j(i_k))$ we do nothing (i.e., ties are also possible).
Once we win completely with player $j$ (i.e., $i_j=N_j$) we do not play with him/her anymore.
We lose the whole game if we lose with at least one player, i.e., when $i_j=0$ for some $j=1,\ldots,d$.

\par 
For $d=1$ and  $p_1(j)=p, q_1(j)=q$ we have the classical Gambler's ruin problem. 
Many one-dimensional generalizations of this game were considered, e.g., \cite{Lefebvre2008} studied the case of some specific
sequences of $p_1(j), q_1(j)$, later this was extended to any  $p_1(j), q_1(j)$ in  \cite{El-Shehawey2009}. Variations of classical problem with 
ties allowed, i.e., $p_1(j)=p_1, q_1(j)=q_1, p_1+q_1<1$ were considered in, e.g., \cite{Lengyel2009}, \cite{Lengyel2009a} (the latter one considers so called \textsl{conditional}
version of the problem). Some of the articles studied both, the ruin probability and  duration of the game, whereas most papers studied only duration of the game. 
Some generalizations to higher a dimension $d\geq 1$ were studied in \cite{Rocha2004}, \cite{Kmet2002}.

 
We will describe the game more formally as a Markov chain $Z'$ with two absorbing states. 
 The state space is $\E'=\{(i_1,\ldots, i_d): 1\leq i_j\leq N_j, 1\leq j\leq d\}\cup\{-\infty\} $ (where $-\infty$ means we \textsl{lose}).
 For convenience denote $p_j(N_j)=q_j(N_j)=0=p_j(0)=q_j(0), j=1,\ldots,d$. Assume  that for all $i_j\in\{1,\ldots,N_j\}, j\in\{1,\ldots,d\}$ we have  
 $p_j(i_j)>0, q_j(i_j)>0 $ and  
 $  \sum_{k=1}^d (p_k(i_k)+q_k(i_k)) \leq 1. $
 With some abuse of notation, we will sometimes write $ (i'_1,\ldots,i'_d)=-\infty$. 
  The transitions of the described chain are following: \medskip \par 
  \noindent
  $\displaystyle \PP_{Z'}((i_1,\ldots,i_d),(i'_1,\ldots,i'_d)) =$
\begin{equation}\label{eq:PZp} 
 \left\{ 
 \begin{array}{llllllll}
  p_j(i_j) & \textrm{if} & i_j'=i_j+1, i_k'=i_k, k\neq j, \\[3pt]
  q_j(i_j) & \textrm{if} & i_j'=i_j-1, i_k'=i_k, k\neq j, \\[3pt]
  \sum_{j:i_j=1}q_j(1) & \textrm{if} & (i'_1,\ldots,i'_j)=-\infty, \\[3pt]
  1-\sum_{k=1}^d (p_k(i_k)+q_k(i_k))& \textrm{if} & i_j'=i_j, 1\leq j\leq d,\\[3pt]
  1&\textrm{if} & (i_1,\ldots,i_j)=(i'_1,\ldots,i'_j)=-\infty.    
 \end{array}
 \right.
\end{equation}
The chain, as required, has two absorbing states: $(N_1,\ldots,N_d)$ (we \textsl{win}) and $-\infty$ (we \textsl{lose}).
We will give formulas for the probabilities of winning starting at arbitrary state, i.e., for 
$$\rho((i_1,\ldots,i_d))=P(\tau_{(N_1,\ldots,N_d)} < \tau_{-\infty} | Z'_0=(i_1,\ldots,i_d)),$$
where $\tau_\e:=\inf\{n\geq 0: Z_n=\e\}$.
Our main result is  following

\begin{twr}\label{twr:twr_main}
Consider the  generalized Gambler's ruin problem described above. Then, the probability of winning starting at
$(i_1,\ldots,i_d)$ is given by
\begin{equation}\label{eq:gabmler_rho}
\rho((i_1,\ldots,i_d))=
{
\displaystyle \prod_{j=1}^d \left(  \sum_{n_j=1}^{i_j}  \prod_{r=1}^{n_j-1} \left({q_j(r)\over p_j(r)}\right) \right)
\over
\displaystyle \prod_{j=1}^d \left(  \sum_{n_j=1}^{N_j} \prod_{r=1}^{n_j-1} \left({q_j(r)\over p_j(r)}\right)
\right) }.
\end{equation}
\end{twr}
\noindent Theorem \ref{twr:twr_main} generalizes some  previous one-dimensional cases. For example:
\begin{itemize}
 \item (i)   Assume we have won with  all the players except player $j$. Then, the probability of winning is
 $$\rho((N_1,\ldots N_{j-1},i_j,N_{j+1},\ldots ,N_d))=  
 {
\displaystyle \sum_{n=1}^{i_j} \prod_{r=1}^{n-1} \left({q_j(r)\over p_j(r)}\right)  
\over
\displaystyle \sum_{n=1}^{N_j} \prod_{r=1}^{n-1} \left({q_j(r)\over p_j(r)}\right)   }.$$
This way formula (4.1) from \cite{El-Shehawey2009} is recovered (with some slight modification in notation,
since  authors considered a various versions of reflecting barriers).
%

\item (ii) In addition to (i), let   $p_j(r)=p, q_j(r)=q, r=1,\ldots,N_j$. Then we recover winning probability in the the classical Gambler's ruin problem (with possible ties):
$$\rho((N_1,\ldots N_{j-1},i_j,N_{j+1},\ldots ,N_d))=
 {
\displaystyle \sum_{n=1}^{i_j}  \left({q \over p}\right)^{n-1}  
\over
\displaystyle \sum_{n=1}^{N_j}  \left({q\over p}\right)^{n-1}  }
=\left\{ 
\begin{array}{lllllll}
\displaystyle {1-\left({q \over p}\right)^{i_j} \over 1-\left({q \over p}\right)^{N_j} } & \mathrm{if} \ p \neq q, \\[20pt]
\displaystyle {i_j\over N_j} & \mathrm{otherwise}.
\end{array}
\right.
$$

\item (iii) \textsl{(Homogeneity case)}. Assume that for all $j=1,\ldots,d$  we have $p_j(r)=p_j, q_j(r)=q_j, r=1,\ldots,N_j$.
Define 
$$\mathcal{S}_j:=\left\{ 
\begin{array}{lll}
 1 & \mathrm{ if\ }   p_j=q_j,\\[8pt]
 0 & \mathrm{otherwise}.
\end{array}\right. 
$$
Then we have
$$
\rho((i_1,\ldots, i_d))
=\prod_{j=1}^d \left\{ \left({1-\left({q_j \over p_j}\right)^{i_j} \over 1-\left({q_j \over p_j}\right)^{N_j} }\right)\mathcal{S}_j
+{i_j\over N_j}(1-\mathcal{S}_j)\right\},$$
which is a multidimensional generalization of classical Gambler's ruin problem.
Of course we obtain the same probabilities if only ratios ${ q_j(i_j)\over p_j(i_j)}$ are constant, e.g., for the following 
spatially nonhomogeneous case 
$$ p_j(i_j) ={p_j\over 2\rho i_j +1}, \qquad q_j(i_j)={q_j\over 2\rho i_j+1},$$
which is thus a multidimensional generalization of cases considered in \cite{El-Shehawey2009}  and in \cite{Lefebvre2008}.
In the latter article only symmetric case corresponding to $p_j=q_j=1/2$   was considered.
\end{itemize}

\section{Tools: Siegmund duality and antiduality}\label{sec:siegmund}

We shortly  recall notion of Siegmund duality, it's applications to studying absorption probabilities and result
concerning existence of Siegmund dual from \cite{Lorek2015_Siegmund_duality}. Let $X$  be a discrete-time Markov chain with transition matrix $\PP_X$ and  finite state space 
$\EE=\{\e_1,\ldots,\e_M\}$ partially ordered by $\preceq$ with unique minimal element $\e_1$ and unique maximal element $\e_M$. Assume it is ergodic with the stationary distribution $\pi$.
For $A\subseteq \EE$ define $\PP_X(\e,A):=\sum_{\e'\in A}\PP_X(\e,\e')$ and similarly $\pi(A):=\sum_{\e\in A}\pi(\e)$. Define also $\{\e\}^\uparrow:=\{\e'\in\EE: \e\preceq \e'\}$,
$\{\e\}^\downarrow:=\{\e'\in\EE: \e'\preceq \e \}$ and $\delta(\e,\e')=\mathbf{1}(\e,\e')$.
We say that Markov chain $Z$ with
transition matrix $\PP_Z$ is the 
\textbf{Siegmund dual} of $X$ if 
\begin{equation}\label{eq:Siegmund_duality}
 \forall(\e_i,\e_j\in \EE)\ \forall(n\geq 0) \quad \PP^n_X(\e_i, \{\e_j\}^\downarrow) = \PP^n_Z(\e_j,\{\e_i\}^\uparrow).
\end{equation}
Note that we can find a matrix fulfilling (\ref{eq:Siegmund_duality}) which is  substochastic, since we may have for some $\e_j$ that  $\sum_{\e_i} \PP_Z(\e_j,\e_i)<1$.
In a similar way as Siegmund \cite{Siegmund1976} did (he considered only linear ordering), we  add then one extra absorbing state,
say $-\infty$ (called a \textsl{coffin state}). Denote the resulting matrix by $\PP_{Z'}$ and define $\PP_{Z'}(\e_j,-\infty )=1-\sum_{\e_i} \PP_{Z'}(\e_j,\e_i)$, $\PP_Z(-\infty, \e_j)=\delta(-\infty,\e_j)$
and $\PP_{Z'}(\e,\e_2)=\PP_Z(\e,\e_2)$ otherwise.
Note  that  (\ref{eq:Siegmund_duality}) implies that $\e_M$ is an absorbing state, thus $Z'$ has two absorbing states.
Taking limits as $n\to\infty$ on both sides of (\ref{eq:Siegmund_duality}) we have 
\begin{equation}\label{eq:siegm_pi}
\pi(\{\e_j\}^\downarrow) = \lim_{n\to\infty} \PP^n_{Z'}(\e_j,\{\e_i\}^\uparrow) = P(\tau_{\e_M}<\tau_{\e_{-\infty}} |  Z'_0=\e_j),
\end{equation}
where $\tau_\e=\inf\{n: Z'_n=\e\}$.
This way the stationary distribution of ergodic chain is related to the absorption  of its Siegmund dual.
\par 
For partial ordering $\preceq$ define $\C(\e,\e') = \mathbf{1}(\e\preceq\e')$. Such matrix is always invertible, and its inverse $\C^{-1}$ is often 
denoted by $\mu$ (what we use throughout the paper) and called \textsl{the M\"obius function of ordering $\preceq$}.
Note that (\ref{eq:Siegmund_duality}) for $n=1$ can be written as  

 \begin{equation}\label{eq:Siegmund_duality2}
 \PP_X\C=\C\PP_Z^T.
\end{equation}
The main result of \cite{Lorek2015_Siegmund_duality} is that for given partial ordering $\preceq$ the Siegmund dual chain 
exists if and only if  $X$ is M\"obius monotone (see  also \cite{Lorek2012d} for more details on this monotonicity).
In such a case, the Siegmund dual on $\EE'=\EE\cup\{-\infty\}$ has transitions outside coffin state given by
 \begin{equation}\label{eq:Siegmund_dual2}
 \PP_{Z'}=(\C^{-1}\PP_X\C)^T
\end{equation}
(the nonnegativity of which  is the definition of M\"obius monotonicity of $X$).
The natural application is in studying stationary distribution of a chain $X$ (e.g., its asymptotics): calculate Siegmund dual 
and then its probability of being eventually absorbed in $\e_M$.
However, we can reverse the process starting with a chain $Z'$ with two absorbing states (we \textsl{win} or we \textsl{lose}).
Assume its state space is $\EE'=\{-\infty\}\cup \{\e_1,\ldots,\e_M\}$ with absorbing states $-\infty$ and $\e_M$. Denote $\EE:=\{\e_1,\ldots,\e_M\}$.
\smallskip
\par
\noindent The procedure is then the following: 
\begin{itemize}
 \item[1)]  remove state $-\infty$ obtaining substochastic matrix $\PP_Z$;
 \item[2)] introduce some partial ordering $\preceq$ expressed by matrix $\C$ such that $\e_M$ is a unique maximal element;
 \item[3)] calculate transitions of \textbf{Siegmund antidual}  chain $X$ from  
 (\ref{eq:Siegmund_duality2}) calculating $\PP_X=\C\PP_Z^T \C^{-1}$;
 \item[4)] if the resulting matrix $\PP_X$ has a stationary measure $\pi$ such that $\forall(\e\in\EE)$ $ \lim_{n\to\infty}\PP_X^n(\e,\cdot)=\pi(\cdot)$ 
 then  we can calculate absorption probabilities of $Z'$ from  the relation (\ref{eq:siegm_pi})
 (if $\PP_X$ is a stochastic matrix, then $\pi$ is the stationary distribution of the chain related to this matrix).
\end{itemize}

%
\noindent
The details are in the following theorem.
\begin{twr}\label{th:main_antiduality}
Let $Z'$ be a Markov chain with transition matrix $\PP_{Z'}$ on $\E'=\{-\infty\}\cup\{e_1,\ldots,\e_M\}=:\{-\infty\}\cup\EE$ with two absorbing states $-\infty$ and $\e_M$.
Consider substochastic kernel $\PP_Z$ which is the matrix $\PP_{Z'}$ with row and column corresponding to $-\infty$ removed.
Fix some partial ordering $\preceq$ on $\EE$ expressed by matrix $\C$  such that $\e_M$ is a unique maximal state.
Calculate 
 \begin{equation}\label{eq:Siegmund_antiduality}
\PP_X=\C\PP_Z^T \C^{-1}.
\end{equation}
The resulting $\PP_X$ has the property that $\forall(\e\in\EE)$ $\sum_{\e_2\in\EE}\PP_X(\e,\e_2)=1$. 
Assume that there exists invariant measure $\pi$ fulfilling:
$$ \forall(\e_2\in\EE) \lim_{n\to\infty}\PP_X^n(\e_2,\e)=\pi(\e), \qquad \sum_{\e\in\EE}\pi(\e)=1.$$
Then we have $\rho(\e'):=P(\tau_{\e_M}<\tau_{-\infty } | Z'_0=\e')=\pi(\{\e'\}^\downarrow)$ (i.e., relation  (\ref{eq:siegm_pi}) holds).
%
\end{twr}

\begin{remark}\rm 
If resulting $\PP_X$ in  (\ref{eq:Siegmund_antiduality}) is a stochastic matrix of ergodic chain, say $X$, then 
$\pi$ is its stationary distribution. Moreover, it time reversal is M\"obius monotone with respect to $\preceq$.
\end{remark}
\begin{remark}\rm 
If the resulting $\PP_X$ has negative entries it does not have a real probabilistic interpretation.
However, e.q., in area of quantum mechanics, such \textsl{``distributions''}, called \textsl{negative quasi-probabilities} are quite common
and natural in this context. This notion was already introduced in \cite{Wigner1932}, where author writes:
\par 
 \begin{center}
  \begin{minipage}[t]{0.85\textwidth}
\textsl{``[...] cannot be really interpreted as the
simultaneous probability for coordinates and momenta, as is clear from the
fact, that it may take negative values. But of course this must not hinder the
use of it in calculations as an auxiliary function which obeys many relations
we would expect from such a probability.
``}
\end{minipage}  
 \end{center}
\par \noindent 
For some recent connections of negative quasi-probability and quantum computations see \cite{Veitch2012}.
\end{remark}
\medskip\par

\begin{proof}[Proof of Theorem \ref{th:main_antiduality}]
The main sketch of the proof was essentially given  before the theorem. 
The only thing which may be not clear  is that for all $\e$ we have $\sum_{\e'} \PP_X(\e,\e')=1$. 
Let us calculate
 $$\sum_{\e'} \PP_X(\e,\e')= \sum_{\e'} (\C\PP_Z^T \C^{-1})(\e,\e')= \sum_{\e'} \sum_{\e_2}  (\C\PP_Z^T)(\e,\e_2) \C^{-1}(\e_2,\e').$$
 We have 
 $$
 \begin{array}{rcl}
 \displaystyle (\C\PP_Z^T)(\e,\e_2) & = &  \displaystyle \sum_{\e_3} \C(\e,\e_3) \PP_Z^T(\e_3,\e_2)  =  \displaystyle  \sum_{\e_3} \C(\e,\e_3) \PP_Z(\e_2,\e_3) \\[16pt]
            &  = & \displaystyle   \sum_{\e_3\succeq \e}  \PP_Z(\e_2,\e_3)=\PP_Z(\e_2,\{\e\}^\uparrow), \\
 \end{array}
 $$
 thus
 $$
 \begin{array}{rcl}
 
  \displaystyle  \sum_{\e'} \PP_X(\e,\e') & = &  \displaystyle  \sum_{\e'} \sum_{\e_2} \PP_Z(\e_2,\{\e\}^\uparrow) \mu(\e_2,\e')
  =\sum_{\e'}  \mu(\e_2,\e') \sum_{\e_2} \PP_Z(\e_2,\{\e\}^\uparrow) \\[10pt]
   &  \stackrel{(*)}{=} &   \displaystyle \PP_Z(\e_M,\{\e\}^\uparrow)=1. \\
 \end{array}
 $$
%
%
%
%
%
%
  In $(*)$ we used the fact, that for any partial order with unique maximal element $\e_M$, the M\"obius function 
  fulfills $\forall(\e\in\E)\ \sum_{\e_j} \C^{-1} (\e,\e_j)=\mathbf{1}(\e=\e_M)$.  
To see this consider column of $\C^{-1}$ corresponding to state $\e_M$ after applying first elementary column operation of Gauss-Jordan elimination.
\end{proof}
\noindent
\section{Toy example: Cat Eats Mouse Eats Cheese}\label{sec:toy_example}

\noindent
Before proceeding to the proof of the main result on generalized Gambler's ruin problem (i.e., Theorem \ref{twr:twr_main}) we give a 5-state example.
The reason for this is that we wanted to present an example having the resulting matrix $\PP_X$ with  negative entries.
The example is taken from \cite{Bremaud1999} (Example 3.2 \textsl{Cat Eats Mouse Eats Cheese}, 
where the answer is easily calculated using first-step analysis):

\smallskip\par 
 \begin{center}
  \begin{minipage}[t]{0.85\textwidth}
\textsl{``A merry mouse moves in a maze. If it is at time $n$ in a room with $k$ adjacent rooms, it will be at time $n+1$ in one of the $k$ adjacent rooms,
choosing one at random, each with probability ${1\over k}$. A fat lazy cat remains all the time in a given room, and a piece of cheese waits for the mouse in another room 
(see Fig. \ref{fig:mouse}). The cat is not completely lazy: If the mouse enters the room inhabited by the cat, the cat will eat it.
What is the probability that the mouse ever gets to eat the cheese when starting from room 1, the cat and the cheese being in rooms 3 and 5, respectively?
``}
\end{minipage}  
\end{center}
 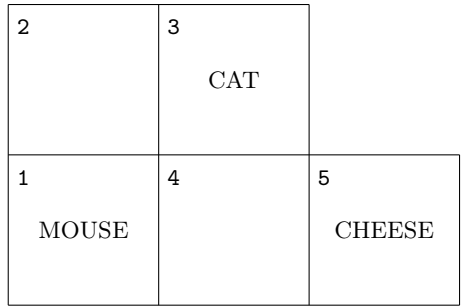
\begin{figure}[h!]
\centering
\begin{tikzpicture}
\draw (0,0) -- (0,4);
\draw (0,0) --(6,0);
\draw (2,0) -- (2,4);
\node at (1,1) {\small MOUSE};
\draw (4,0) -- (4,4);
\node at (5,1) {\small CHEESE};
\draw (6,0) -- (6,2);

\draw (0,2) --(6,2);
\draw (0,4) --(4,4);
\node at (3,3) {\small CAT};
\node at (2.2,3.7) {\small \texttt{3}};
\node at (0.2,3.7) {\small \texttt{2}};
\node at (0.2,1.7) {\small \texttt{1}};
\node at (2.2,1.7) {\small \texttt{4}};
\node at (4.2,1.7) {\small \texttt{5}};
\end{tikzpicture}
\caption{Maze, mouse, and murder} \label{fig:mouse}
\end{figure}
 
 \noindent
Essentially
we have $\EE'=\{1,2,3,4,5\}$ and 
$$  
 \PP_{Z'}=
 \left( 
 \begin{array}{llllllll}
 0 & {1\over 2} & 0 & {1\over 2} & 0 \\[2pt]
  {1\over 2} & 0 & {1\over 2} & 0 & 0  \\[2pt]
  0 & 0 & 1 & 0 & 0 \\[2pt]
  {1\over 3} & 0 & {1\over 3} & 0 & {1\over 3} \\[2pt]
  0 & 0 & 0 & 0 & 1 
 \end{array}
 \right), \ 
 \PP_Z=
 \left( 
 \begin{array}{llllllll}
 0 & {1\over 2} &  {1\over 2} & 0 \\[2pt]
  {1\over 2} & 0 &  0 & 0  \\[2pt]
  {1\over 3} & 0 &  0 & {1\over 3} \\[2pt]
  0 & 0 & 0 & 1 
 \end{array}
 \right), \ \C =
 \left( 
 \begin{array}{llllllll}
 1 & 1 & 1 & 1 \\[2pt]
 0 & 1 & 0 & 1 \\[2pt]
 0 & 0 & 1 & 1 \\[2pt]
 0 & 0 & 0 & 1 
 \end{array}
 \right),
$$
where $\PP_{Z'}$ is the original \textsl{Cat Eats Mouse Eats Cheese} matrix, $\PP_Z$ is the matrix with state 3 removed ($\E=\{1,2,4,5\}$) and $\C$ represents the ordering we introduced 
ordering with $5$ being a maximal state. We are to calculate 
$$\rho(j)=P(\tau_5<\tau_3 | Z_0'=j), j=1,2,4,5.$$
Calculating $\PP_X$ from (\ref{eq:Siegmund_antiduality}) and its stationary measure gives
$$\PP_X=\C\PP_Z^T\C^{-1}=
 \left( 
 \begin{array}{rrrrrrr}
 1 & -{1\over 2} &  -{1\over 3} & {5\over 6} \\[4pt]
 {1\over 2} & -{1\over 2} &  -{1\over 6} & {7\over 6}  \\[4pt]
 {1\over 2} & -{1\over 2} &  -{1\over 6} & {7\over 6}  \\[4pt]
 0 & 0 & {1\over 3} & {2\over 3} 
 \end{array}
 \right), \ (\pi(1),\pi(2),\pi(4),\pi(5))=\left({2\over 7}, -{1\over 7}, {1\over 7}, {5\over 7}\right). 
$$ 
From   (\ref{eq:siegm_pi}) we have $(\rho(1),\rho(2),\rho(4),\rho(5))=\left({2\over 7}, {1\over 7}, {1\over 7}, 1\right)$.
\section{Proof of Theorem \ref{twr:twr_main}}\label{sec:proof_of_main}
From matrix $\PP_{Z'}$ given in (\ref{eq:PZp}) we remove state $-\infty$ obtaining the following substochastic matrix
on $\EE=\{1,2,\ldots,N_1\} \times \cdots \times \{1,2,\ldots,N_d\}$
$$
 \PP_Z((i_1,\ldots,i_d),(i'_1,\ldots,i'_d)) = 
 \left\{ 
 \begin{array}{llllllll}
  p_j(i_j) & \textrm{if} & i_j'=i_j+1, i_k'=i_k, k\neq j, \\[5pt]
  q_j(i_j) & \textrm{if} & i_j'=i_j-1, i_k'=i_k, k\neq j, \\[5pt]
  1-\sum_{k=1}^d (p_k(i_k)+q_k(i_k))& \textrm{if} & i_j'=i_j, 1\leq j\leq d.\\      
 \end{array}
 \right.
 $$
Consider the coordinate-wise ordering: $ (i_1,\ldots,i_d) \preceq (i'_1,\ldots,i'_d) \ \textrm{ iff } \ i_j\leq i_j', j=1,\ldots,d.$
The state $\e_M:=(N_1,\ldots,N_d)$ is a unique maximal one.
Directly from Proposition 5 in   \cite{Rota64} we find the
corresponding M\"obius function
$$
\mu((i_1,\ldots,i_d), (i_1+r_1,\ldots,i_d+r_d))
=\left\{\begin{array}{ll}
  \displaystyle(-1)^{\sum_{k=1}^n r_k}  &  r_j\in\{0,1\}, \ i_j+r_j\leq N_j,\  j=1,\ldots,d \\[8pt]
0 & \textrm{otherwise}. \\
\end{array}\right.
$$
We have: $\PP_X((i_1,\ldots,i_d),(i'_1,\ldots,i'_d))  =  \C\PP_Z^T\C^{-1}((i_1,\ldots,i_d),(i'_1,\ldots,i'_d))$
$$
\begin{array}{lll} 
\displaystyle=\sum_{(i^{(2)}_1,\ldots,i^{(2)}_d)\preceq (i'_1,\ldots,i'_d)} \mu((i^{(2)}_1,\ldots,i^{(2)}_d),(i'_1,\ldots,i'_d))\PP_Z((i^{(2)}_1,\ldots,i^{(2)}_d),\{(i_1,\ldots,i_d)\}^\uparrow). 
\end{array}
$$
Let $\mathbf{s}_j=(0,\ldots,0,1,0,\ldots,0)$, where $1$ is on position $j$.
Consider the following cases: \smallskip\par 
\noindent $\bullet$ Case $ (i'_1,\ldots,i'_d)=(i_1,\ldots,i_d)-\es_j$. Then we have $\PP_X((i_1,\ldots,i_d),(i'_1,\ldots,i'_d))=$
\begin{eqnarray*}
 = & & \sum_{(i^{(2)}_1,\ldots,i^{(2)}_d)\preceq (i_1,\ldots,i_d)-\es_j} \mu((i^{(2)}_1,\ldots,i^{(2)}_d),(i_1,\ldots,i_d)-\es_j)\PP_Z((i^{(2)}_1,\ldots,i^{(2)}_d),\{(i_1,\ldots,i_d)\}^\uparrow)\\
 = & &\mu((i_1,\ldots,i_d)-\es_j,(i_1,\ldots,i_d)-\es_j) \PP_Z((i_1,\ldots,i_d)-\es_j,\{(i_1,\ldots,i_d)\}^\uparrow)\\[3pt]
  & + & \sum_{(i^{(2)}_1,\ldots,i^{(2)}_d)\preceq (i_1,\ldots,i_d)-\es_j \atop (i^{(2)}_1,\ldots,i^{(2)}_d)\neq (i_1,\ldots,i_d)-\es_j} \mu((i^{(2)}_1,\ldots,i^{(2)}_d),(i_1,\ldots,i_d)-\es_j)\PP_Z((i^{(2)}_1,\ldots,i^{(2)}_d),\{(i_1,\ldots,i_d)\}^\uparrow). \\
 = & &1\cdot \PP_Z((i_1,\ldots,i_d)-\es_j,(i_1,\ldots,i_d))+0=q_j(i_j-1).
\end{eqnarray*}
\noindent 
\noindent $\bullet$ Case $ (i'_1,\ldots,i'_d)=(i_1,\ldots,i_d)+\es_j$. We have $\PP_X((i_1,\ldots,i_d),(i'_1,\ldots,i'_d))=$
\begin{eqnarray*}
 = & & \sum_{(i^{(2)}_1,\ldots,i^{(2)}_d)\preceq (i_1,\ldots,i_d)+\es_j} \mu((i^{(2)}_1,\ldots,i^{(2)}_d),(i_1,\ldots,i_d)+\es_j)\PP_Z((i^{(2)}_1,\ldots,i^{(2)}_d),\{(i_1,\ldots,i_d)\}^\uparrow)\\
 = & &\mu((i_1,\ldots,i_d)+\es_j,(i_1,\ldots,i_d)+\es_j) \PP_Z((i_1,\ldots,i_d)+\es_j,\{(i_1,\ldots,i_d)\}^\uparrow)\\[3pt]
  & + &\mu((i_1,\ldots,i_d),(i_1,\ldots,i_d)+\es_j) \PP_Z((i_1,\ldots,i_d),\{(i_1,\ldots,i_d)\}^\uparrow)\\[3pt]
  & + & \sum_{r=1,\ldots,d\atop r\neq j} \mu((i_1,\ldots,i_d)-\es_r+\es_j,(i_1,\ldots,i_d)+\es_j)\PP_Z((i_1,\ldots,i_d)-\es_r+\es_j,\{(i_1,\ldots,i_d)\}^\uparrow) \\
  & + & \sum_{r=1,\ldots,d\atop r\neq j} \mu((i_1,\ldots,i_d)-\es_r,(i_1,\ldots,i_d)+\es_j)\PP_Z((i_1,\ldots,i_d)-\es_r,\{(i_1,\ldots,i_d)\}^\uparrow) \\
= & &(-1)^0 \cdot \PP_Z((i_1,\ldots,i_d)+\es_j,\{(i_1,\ldots,i_d)\}^\uparrow) + (-1)^1 \cdot  \PP_Z((i_1,\ldots,i_d),\{(i_1,\ldots,i_d)\}^\uparrow) \\[3pt]
& + & \sum_{r=1,\ldots,d\atop r\neq j} \left[(-1)^1\cdot \PP_Z((i_1,\ldots,i_d)-\es_r+\es_j,\{(i_1,\ldots,i_d)\}^\uparrow) \right.\\
&  & \quad \qquad + \left. (-1)^2\cdot \PP_Z((i_1,\ldots,i_d)-\es_r,\{(i_1,\ldots,i_d)\}^\uparrow)\right]\\
  = & & 1-\sum_{r=1,\ldots,d\atop r\neq j} q_r(i_r)-\left(1-\sum_{r=1} q_r(i_r)\right)+\sum_{r=1,\ldots,d\atop r\neq j}\left(-1+1\right)p_r(i_r)=q_j(i_j).
\end{eqnarray*}
\noindent $\bullet$ Case $ (i'_1,\ldots,i'_d)=(i_1,\ldots,i_d)$. We have $\PP_X((i_1,\ldots,i_d),(i_1,\ldots,i_d))=$
\begin{eqnarray*}
 = & & \sum_{(i^{(2)}_1,\ldots,i^{(2)}_d)\preceq (i_1,\ldots,i_d)} \mu((i^{(2)}_1,\ldots,i^{(2)}_d),(i_1,\ldots,i_d))\PP_Z((i^{(2)}_1,\ldots,i^{(2)}_d),\{(i_1,\ldots,i_d)\}^\uparrow)\\
 = & &\mu((i_1,\ldots,i_d),(i_1,\ldots,i_d)) \PP_Z((i_1,\ldots,,i_d),\{(i_1,\ldots,i_d)\}^\uparrow)\\[3pt]
  & + & \sum_{r=1,\ldots,d} \mu((i_1,\ldots,i_d)-\es_r,(i_1,\ldots,i_d))\PP_Z((i_1,\ldots,i_d),\{(i_1,\ldots,i_d)\}^\uparrow) \\
  = & &(-1)^0 \cdot \PP_Z((i_1,\ldots,i_d),\{(i_1,\ldots,i_d)\}^\uparrow) + \sum_{r=1,\ldots,d} (-1)^1\cdot\PP_Z((i_1,\ldots,i_d),\{(i_1,\ldots,i_d)\}^\uparrow) \\
  = & & 1-\sum_{r=1,\ldots,d } q_r(i_r)-\sum_{r=1,\ldots,d} p_r(i_r-1)=1-\sum_{r=1,\ldots,d } ( p_r(i_r-1)+q_r(i_r)).
\end{eqnarray*}
These were the only nonzero entries of $\PP_X$. Summarizing  we have
\medskip
\par\noindent 
$\PP_X((i_1,\ldots,i_d),(i'_1,\ldots,i'_d))=$
$$
 \left\{ 
 \begin{array}{llllllll}
  q_j(i_j) & \textrm{if} & i_j'=i_j+1, i_k'=i_k, k\neq j, \\[5pt]
  p_j(i_j-1) & \textrm{if} & i_j'=i_j-1, i_k'=i_k, k\neq j, \\[5pt]
  1-\sum_{k=1}^d (p_k(i_k-1)+q_k(i_k))& \textrm{if} & i_j'=i_j, 1\leq j\leq d.\\      
 \end{array}
 \right.
$$
Assume that for all $(i_1,\ldots,i_d)$ we have   $\PP_X((i_1,\ldots,i_d),(i_1,\ldots,i_d))\geq 0$.
Then, these  are the transitions of a closed network with $d$ independent servers: being at state $(i_1,\ldots,i_d)$ the  arrival to server $j$ is $q_j(i_j)$ and the departure is $p_j(i_j-1)$.
Its stationary distribution is following
\begin{equation}\label{eq:antid_pi}
\pi((i_1,\ldots,i_d))=
{
\displaystyle \prod_{j=1}^d \left( \prod_{r=1}^{i_j-1} \left({q_j(r)\over p_j(r)}\right) \right)
\over
\displaystyle \prod_{j=1}^d \left(\sum_{n_j=1}^{N_j} \prod_{r=1}^{n_j-1} \left({q_j(r)\over p_j(r)}\right) \right) }.
\end{equation}
However, if for some  entry of $\PP_X$ is negative, then of course the matrix does not 
represent any Markov chain, but with $\pi$ given in (\ref{eq:antid_pi})   we still have that $\pi\PP_X=\pi$ and $\lim_{n\to\infty}\PP_X^n((i_1,\ldots,i_d),\cdot)=\pi(\cdot)$ for 
every $(i_1,\ldots,i_d)\in\E$. \par 
Thus, because of Theorem \ref{th:main_antiduality} equality   (\ref{eq:siegm_pi}) holds in any case and we obtain  (\ref{eq:gabmler_rho}) what finishes the proof.

%
%

\section*{Acknowledgments}
Author thanks  Zbigniew Palmowski for helpful suggestions and discussions.
Partial support by the project RARE-318984, a Marie Curie IRSES Fellowship within the 7th European Community Framework Programme, is kindly acknowledged.






\end{document}